\documentclass[11pt]{article}
\usepackage{amsmath}
\usepackage{amsthm}
\usepackage{amssymb}

\newtheorem{theorem}{Theorem}
\newtheorem{corollary}{Corollary}

\newtheorem{lemma}{Lemma}
\newtheorem{definition}{Definition}

\newtheorem{remark}{Remark}

\begin{document}

\begin{center}\LARGE
%%% Insert the title of your talk here %%%

\textbf{A class of symmetric association schemes as inclusion of biplanes}

\bigskip\large

%%% Insert your first name and family name here %%%
Ivica Martinjak\\
University of Zagreb

%%% Insert your institution here %%%

%%% Insert you city and country here %%%
\end{center}

%%% Give the abstract of your talk here %%%

\begin{abstract} 
Let ${\cal B}$ be a nontrivial biplane of order $k-2$ represented by symmetric canonical incidence matrix with trace $1+ \binom{k}{2}$. We proved that ${\cal B}$ includes a partially balanced incomplete design with association scheme of three classes. Consequently, these structures are symmetric, having $2k-6$ points. While it is not known whether this class is finite or infinite, we show that there is a related superclass with infinitely many representatives. 
%While it is not known wheather this family has finitely or infinitely many memers, we show that it is a subclass of an infinite family of 
\end{abstract}

%\noindent {\bf AMS Mathematical Subject Classifications:} 05A17, 11P84

\section{Introduction}
\label{Sec1}

Let $X$ be a nonempty finite set. A {\it symmetric association scheme} ${\cal A}$ of $d$ classes on $X$ is a sequence of relations $R_0, R_1,...,R_d \subset X \times X$, satisfying 

\begin{description}
\item {\it i)} $R_0=\{(x,x): x \in X\}$ 
\item {\it ii)} $X \times X =  R_0 \cup R_1 \cup ... \cup R_d$ and $R_i \cap R_j={\o}$, $\forall i \neq j$
\item {\it iii)} $(x,y) \in R_i \Leftrightarrow (y,x) \in R_i$, $i \in\{0,1,...,d\}$
\item {\it iv)} $\forall h,i,j \in \{0,1,...,d\}$ and $\forall x,y \in X$ such that $(x,y) \in R_h$ the number $$p_{ij}^h:=|\{z \in X : (x,z) \in R_i, (z,y) \in R_j\}|$$ is well defined i.e. depends only on $h,i,j$ and not on $x$ and $y$.
\end{description}

We say that points $(x,y) \in R_i$ are {\it $i$-th associates}. The numbers $p_{ij}^h$ are the {\it intersection numbers} of ${\cal A}$. The relation $R_i$ is called $i$-th {\it associate class} and can be represented by the $i$-th {\it associate matrix} $A_i$. The matrix $A_i$ is a (0,1)-matrix whose rows and columns are indexed by the elements of $X$ and whose $(x,y)$ entry defines whether $(x,y)\in R_i$ or not. Matrices $$A_0,A_1,...,A_d$$ have several interesting properties, forming the vector space that is also a matrix algebra. This associative algebra is called the {\it Bose-Mesner algebra of ${\cal A}$}. 
Sum of these matrices equals the all-1 matrix $J_{|X|}$ of order $|X|$,
$$
\sum_i A_i=J_{|X|}.
$$
The scheme ${\cal A}$ is also represented by a matrix ${\cal R}=[r_{xy}]$ of order $|X|$, defined on the way $$r_{xy}=i \Leftrightarrow (x,y) \in R_i,$$ where $i \in \{0,1,...,d\}.$
It is worth mentioning that there are several classes of association schemes, including {\it Hamming} and {\it Johnson schemes}. 

An {\it incidence structure} is a triple ${\cal I}=({\cal P}, {\cal L}, I)$ where ${\cal P}=\{p_1,...,p_v\}$ is a set of {\it points}, the elements of a set ${\cal L}=\{L_1,...,L_b\}$ are called {\it lines} or {\it blocks} and $I \subseteq {\cal P} \times {\cal L}$ is an {\it incidence relation}. If every point is incident with the same number $r$ of lines the structure ${\cal I}$ is called {\it regular} with the degree of regularity equal to $r$, denoted $d(p)=r$. Similarly, the structure is {\it uniform}, with the degree of uniformity equal to $k$ if it holds $d(L)=k, \enspace \forall L \in {\cal L}$. Another special characteristics of an incidence structure is {\it balance}, the property of having every $t$ points on the same number $\lambda$ of lines. These properties are not independet, uniformity and balance imply regularity. The order of such a structure is defined as the difference $r-\lambda$. Naturally, an incidence structure is represented by incidence matrix, whose entries define the incidence relation $I$.

The number of points and lines of a regular and uniform structure ${\cal I}$ are related on the way 
\begin{eqnarray} \label{Rel1}
v \cdot r=b \cdot k.
\end{eqnarray}
Uniform and balanced incidence structure is uniquely determined by the 4-tuple of parameters 
\begin{eqnarray*}
t-(v,k,\lambda).
\end{eqnarray*} 
Sometimes, the longer version $t$-$(v,b,r,k,\lambda)$ is also used to define such structure. In particular, when $t=2$ it holds
$
r (k-1)=\lambda (v-1).
$
Both of the mentioned relations can be proven by means of double counting. The mentioned case $t=2$ is necessary for a remarkable possibility that structure has the same number of points and lines. In that case $r=k$ and such structures are called {\it symmetric}.

Throughout this paper we use standard notation for matrices. In addition to unit matrix $J_{m,n}$, zero and identity matrix with related dimensions are denoted by $0_{m,n}$ and $I_n$, respectively. The matrix $A^T$ is transpose of matrix $A$. Anti-diagonal (0,1)-matrix (anti-cyclic matrix) of order $n$ is denoted by $C_n^{-}$. For given $m \times n$ matrix $A$, $I \subset [m], \enspace J \subset [n]$ a submatrix with entries indexed by $I$ and $J$ is denoted $A_{I,J}$. In particular, when $J= [n]$, this index set possibly will be omited. Similarly, index $(i)$ stands for $I,J=\{i\}$.

\begin{definition}\label{Def1}
A {\it biplane} is a symmetric incidence structure having $k$ points on every of $v$ lines, where any two lines intersect in two points. 
\end{definition}

According to this definition, for biplanes holds $t=\lambda=2$ and such a structure is determined by parameters 
\begin{eqnarray*}
2-(1+\binom{k}{2},k,2).
\end{eqnarray*}

Biplanes are very regular structures with notable algebraic and combinatorial properties \cite{Cam}. In addition to incidence matrix, a biplane can also be represented as a labeled complete graph. Let ${\cal B}_{na}$, ${\cal B}_{nb}$,... be biplanes of order $n$ in respect to their {\it automorphism group order}, in ascending order.

It is known that the first $k$ rows $M_{[k]}$ of biplane's incidence matrix $M$ is uniquely determined up to isomorphism, as follows,
\begin{eqnarray}
M_{[k]}=
\begin{bmatrix}
J_{k,1} & J_{1,k-1} & 0_{1,k-2} & ... & 0_{k-2,1}
\cr  & I_{k-1} & J_{1,k-2} & ... & J_{1,1}
\cr  &  & I_{k-2} & ... & I_1
\end{bmatrix}.
\end{eqnarray}
Furthermore, the first $k$ columns is the transpose of $M_{[k]}$, without loosing generality up to isomorphism. We say that such incidence matrix $M$ of a biplane is in {\it canonical form}. 

\begin{definition}
A {\it partially balanced incomplete design} with $m$ {\it associate classes} ${\cal D}(d)$ is a balanced and uniform incidence structure with $v$-set ${\cal P}$ of points and $b$-set ${\cal B}$ of lines. Any two points that are $i$-th associates appear togather in $\lambda_i$ lines, $i \in \{0,1,...,d\}$. Point $p \in {\cal P}$ is $i$-th associate with $n_i$ points $p' \in {\cal P}$. 
\end{definition}

This definition allows cases $\lambda_i=0$ and $\lambda_i=\lambda_j$. The numbers $v,b,r,k, \lambda_i, \enspace 1 \leq i\leq d$  are parameters  of a partially incomplete design with $d$ associate classes ${\cal D}(d)$. These parameters, usually written as 5-tuple 
\begin{eqnarray*}
2-(v,b,r,k,\lambda_i)
\end{eqnarray*}
satisfy relation \ref{Rel1} as well as relation 
\begin{eqnarray} \label{RelPBD}
\sum_i n_i\lambda_i =r(k-1).
\end{eqnarray}

Recall that a partially balanced incomplete design determines an association scheme, but the converse is not true. In this paper we follow the usuall terminology and use the common name for these two structures. In the same manner we assume that a scheme is symmetric unless otherwise specified, and this is always the case in what follows.

\section{The main result}

We prove that an incidence matrix of a biplane that is symmetric and have the trace equal to the number of points defines a partially incomplete design with 3 associate classes. Design ${\cal D}(3)$ is a subset of such a biplane. Our proof is provided in two steps, corresponding to relations \ref{Rel1} and \ref{RelPBD}. 

When a biplane's incidence matrix $M$ have at least first $3k-5$ entries on the main diagonal equal to 1, then there is a submatrix of $M$ with row and column sum equal to $3$. More precisely, a principal submatrix with this property is $M_{S}$, $$S=\{k+2,...,3k-5\}.$$  This fact, that follows because of the constraint to the third (second) canonical row (column) and the oposite, is expressed in the next lemma. Namely, according to Definition \ref{Def1} the scalar product of any two rows of a biplane's incidence matrix is equal to 2. We also say that any two rows have two 1s in common. The same holds for any two columns

\begin{lemma} \label{Lem1}
Let $M$ be a canonical incidence matrix of a biplane of order $k-2$, with $M[i,i]=1, \enspace i=1,...,3k-5$. Then,
\begin{eqnarray}
\sum_{i=k+2}^{3k-5} M[i,j]= \sum_{i=k+2}^{3k-5} M[j,i]= 3,
\end{eqnarray}
where $j=k+2,...,3k-5$.
\end{lemma}

\begin{proof}
A row $$M_{(i),[v]}\enspace i=k+2,...,2k-2$$ has no one 1s in common on the first $2k-2$ position with the third canonical row $M_{(3),[v]}$ Since the last 1s in $M_{(3),[v]}$ is on $(3k-5)$-th position, $M_{(i),[v]}$ has exactly two 1s on positions $k+2,..., 3k-5$. This two 1s with the diagonal one gives the row-sum 3. 
Columns $$M_{[v],(j)},\enspace j=2k-2,...,3k-5$$ have the analogue constraint towards the second canonical column $M_{[v],(2)}$, meaning that the related column-sum is also 3. 

Analogue reasoning worth for dual rows and columns, which completes the proof.
\end{proof}

The matrix $L_m$ is defined on the way that positions with entries equal to 1 are as follows, while the other entries are 0,
$$
\qquad
L_{n}=\begin{bmatrix}
1 & 1 & &  &
\cr   1& & 1 & & 
\cr   & &  \ddots& & 
\cr   &   &1 & &1
\cr   &   && 1&1
\end{bmatrix}. 
$$

\begin{lemma}\label{Lem2}
Let $A$ be a (0,1)-matrix of order $m$ with row and column sum 2 and without a submatrix $J_2$. Then for every row $A_{(i)}$ there are exactly two other rows having scalar product with $A_{(i)}$ equal to $1$, while with the rest of $m-3$ vectors this product is $0$.
\end{lemma}

\begin{proof}
Up to row and column permutation there is the only one matrix $A$, $$A=L_m.$$ 
Aparently, every row of $A$ has entry equal  to 1 on one common position with exactly two other rows, while with the rest of $m-3$ vectors scalar product is $0$. Permutations of rows and columns preserves scalar product among rows, which completes the proof.
\end{proof}

\begin{theorem} \label{Thm2}
Let ${\cal B}$ be a nontrivial biplane of order $k-2$, represented by symmetric incidence matrix $M$ having the trace equal to the number of points. Then the principal submatrix $M_S$ represents a partially balanced incomplete design with assciation scheme of $3$ classes ${\cal D}(3)$. 
\end{theorem}

\begin{proof} The condition expressed by relation \ref{Rel1} is provided by Lemma \ref{Lem1}. It remains to prove that numbers $n_i$ do not depend on the choise of a point and to prove that relation \ref{RelPBD} holds (with $k=r=3$). More precisely, the statement that every point is 
\begin{description}
\item {\it i)} $1$st associate with $2k-11$ other points,
\item {\it ii)} $2$nd associate with 2 other point,
\item {\it iii)} $3$rd associate with 2 other point, 
\end{description}
proves the theorem. In that case 
\begin{eqnarray}
\label{AScheme_n1}
n_1&=&2k-11\\ 
n_2=n_3&=&2
\end{eqnarray}
and 
\begin{eqnarray}
\lambda_1&=&0\\
\lambda_2&=&1\\
\lambda_3&=&2.
\end{eqnarray}

According to Lemma \ref{Lem2} for a row $M_{S(i)}$, there are two rows $M_{S(j)}$, such that $$M_{S(i)} \cdot M_{S(j)}=1, \enspace i,j=1,...,k-3, \enspace i \neq j,$$
while the scalar products with the rest $k-6$ rows is $0$. The same holds when $ i,j=k-2,...,2k-6$ with $i \neq j$.

In the final step of the proof we demonstrate that $$M_{S(i)}, \enspace i=1,2,...,k-3$$ has scalar product 2 with exactly two rows $$M_{S(j)}, \enspace j=k-2,...,2k-6$$ while the other intersections with these rows are $0$, when $M_S$ is symmetric. Clearly, then the same holds for interchanged values of $i$ and $j$. Consequently, every row of $M_S$ has scalar product 1 with two other rows, scalar product 2 also with two rows and $0$ with the rest of rows.

Define the matrix $D_n$ as follows,

\begin{eqnarray}
D_{n}=
\begin{bmatrix}
I_{n} & L_{n}
\cr L_{n} & I_{n}
\end{bmatrix}
\end{eqnarray}
Now, the matrix $D_{k-3}$ equals $M_S$ up to row and column permutations. Obviously, this matrix holds the declared scalar products among rows. 

The fact that any finite sequence of row and column permutations of $D_{k-3}$ preserving all 1s on the main diagonal of $M_S$ also preserves the diagonal symmetry completes the proof.
\end{proof}

\begin{remark}
Apart of the trivial biplane (the one of order 1), no one biplane of order $k-2$, $k<6$ admits symmetric canonical incidence matrix. Thus, the right hand side of the relation \ref{AScheme_n1} is at least 1.
\end{remark}

Note that in the proof of theorem, all arguments related to rows hold for columns as well. This means that ${\cal D}(3)$ has equal number of points and lines. Equivalently, the degree of its unifomity $k$ equals the degree of regularity $r$. Because of the same fact, the related association scheme is also symmetric. 

\begin{corollary}
Partially balanced incomplete design with association scheme of three classes ${\cal D}(3)$ is symmetric.
\end{corollary}

Thus, a class of partially balanced incomplete designs and association schemes is naturally associated with biplanes. It follows yet another immediate corollary of Theorem \ref{Thm2}.
\begin{corollary}
There is a family of partially balanced incomplete designs with three associate classes ${\cal D}(3)$, with parameters
\begin{eqnarray} \label{NewClass}
2-(2k-6,3,\lambda_i), \enspace \lambda_i=0,1,2,
\end{eqnarray}
where $k-2$ is the order of a nontrivial biplane admitting symmetric canonical incidence matrix with trace equal to the number of points.
\end{corollary}

We are aware of the fact that even weaker constraint on the trace (see Lemma \ref{Lem1}) leads to the family \ref{NewClass}. However, trace equal to the number of points is a notable property of biplane's canonical incidence matrices \cite{Mar}. 

As an illustration of these results, designs and association schemes ${\cal D}(3)$ with 6 and 16 points are constructed in this work, as it is presented in the next section.

\section{Some representatives}

The matrix $B'$ togather with $M_{[6]}$ and $M_{[6]}^T$, is an incidence matrix of the biplane ${\cal B}_{4c}$. The order of ${\cal B}_{4c}$ is $4$  and the automorphism group order is $|Aut({\cal B}_{4c})| =11520$.

\begin{eqnarray*}
B&=&\begin{bmatrix}
0_{1,3} & J_{1,3}
\cr L_3C_3^{-} & C_3^{-}
\end{bmatrix},\\
B'&=& \begin{bmatrix}
I_4 & B
\cr B^T & I_6+C_6^{-}
\end{bmatrix}.
\end{eqnarray*}

This incidence matrix defines ${\cal D}'(3)$ with $6$ points and $6$ lines. Related submatrix $M'_S$ is an incidence matrix of ${\cal D}'(3)$. Parameters of ${\cal D}'(3)$ are $2$-$(6,6,3,3,\lambda_i)$, or shortly 
\begin{eqnarray*}
2-(6,3,\lambda_i), \enspace  \lambda_i=0,1,2.
\end{eqnarray*} 
According to Theorem \ref{Thm2} it is always the case for a partially balanced incomplete design with association scheme arising from biplane's symmetric canonical incidence matrix having the trace $v$ that the other parameters are $\lambda_1=0, \enspace \lambda_2=1, \enspace \lambda_3=2$ and $n_2=n_3=2$. In this particular case $$n_1=2k-11=1.$$ In other words, the point 1 togather with the point 2 is incident with $1$ line. The same holds for points 1 and 3. Furthermore, point 1 togather with point 4 is not incident with any line. The third case is when 1 togather with some another point is incident with $2$ lines. This worth for pairs $(1,5)$ and $(1,6)$. The equal intersection statistics holds for every of the 5 other points. 

Related association scheme ${\cal A}'$ with 3 classes is represented by the matrix ${\cal R}'$, as follows on the table:

\begin{center}
\begin{tabular}{cccccc}
0&2&2&1&3&3\\
2&0&2&3&1&3  \\
2& 2& 0& 3& 3&1 \\
1& 3& 3& 0& 2&2 \\
3&1 &3 &2 &0 &2 \\
3&3 &1 &2 &2 &0 \\
\end{tabular}
\end{center}

On the other hand, ${\cal A}'$ is represented by matrices $A_i$, $i=0,1,2,3$,
\begin{eqnarray*}
A_0&=&I_6\\
A_1&=&\begin{bmatrix}
0_3 & I_3\\
I_3 & 0_3
\end{bmatrix}\\
A_2&=&\begin{bmatrix}
L_3C_3^{-} &  0_3 \\
 0_3 & L_3C_3^{-}
\end{bmatrix}\\
A_3&=&\begin{bmatrix}
0_3 & L_3C_3^{-}\\
L_3C_3^{-} & 0_3
\end{bmatrix}.
\end{eqnarray*}

Biplanes of the next order 7 do not admit symmetric canonical incidence matrix \cite{Mar}, so no one principal submatrix $M_S$ represents a partially balanced incomplete design. However, we found interesting constelations among these submatrices. Moreover, some of them satisfy sum condition from Lemma \ref{Lem1} (see the left table on Figure \ref{Fig1}, with $\cdot$ instead of 0). There are some submatrices whose the only exception of this condition are starting and ending both rows and columns (example on the right table, Figure \ref{Fig1}). We also found many $M_S$ that are symmetric in respect to the minor diagonal.
\begin{figure}
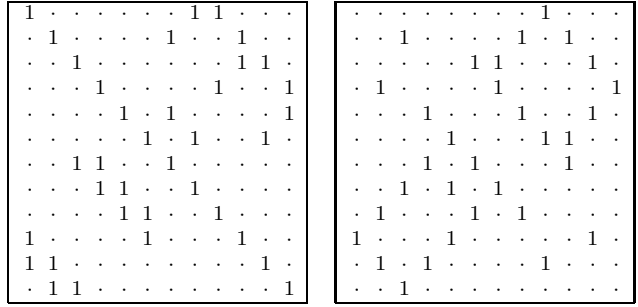

\begin{scriptsize}
\begin{center}
\begin{tabular}{|c@{\hspace{0.55em}}c@{\hspace{0.55em}}c@{\hspace{0.55em}}c@{\hspace{0.55em}}c@{\hspace{0.55em}}c@{\hspace{0.55em}}c@{\hspace{0.55em}}c@{\hspace{0.55em}}c@{\hspace{0.55em}}c@{\hspace{0.55em}}c@{\hspace{0.55em}}c@{\hspace{0.55em}}|c@{\hspace{0.55em}}|c@{\hspace{0.55em}}c@{\hspace{0.55em}}c@{\hspace{0.55em}}c@{\hspace{0.55em}}c@{\hspace{0.55em}}c@{\hspace{0.55em}}c@{\hspace{0.55em}}c@{\hspace{0.55em}}c@{\hspace{0.55em}}c@{\hspace{0.55em}}c@{\hspace{0.55em}}c@{\hspace{0.55em}}|} \cline{1-12} \cline{14-25}
1&	$\cdot$&	$\cdot$&	$\cdot$&	$\cdot$&	$\cdot$&	$\cdot$&	1&	1&	$\cdot$&	$\cdot$&	$\cdot$&	&	$\cdot$&	$\cdot$&	$\cdot$&	$\cdot$&	$\cdot$&	$\cdot$&	$\cdot$&	$\cdot$&	1&	$\cdot$&	$\cdot$&	$\cdot$\\
$\cdot$&	1&	$\cdot$&	$\cdot$&	$\cdot$&	$\cdot$&	1&	$\cdot$&	$\cdot$&	1&	$\cdot$&	$\cdot$&	&	$\cdot$&	$\cdot$&	1&	$\cdot$&	$\cdot$&	$\cdot$&	$\cdot$&	1&	$\cdot$&	1&	$\cdot$&	$\cdot$\\
$\cdot$&	$\cdot$&	1&	$\cdot$&	$\cdot$&	$\cdot$&	$\cdot$&	$\cdot$&	$\cdot$&	1&	1&	$\cdot$&	&	$\cdot$&	$\cdot$&	$\cdot$&	$\cdot$&	$\cdot$&	1&	1&	$\cdot$&	$\cdot$&	$\cdot$&	1&	$\cdot$\\
$\cdot$&	$\cdot$&	$\cdot$&	1&	$\cdot$&	$\cdot$&	$\cdot$&	$\cdot$&	1&	$\cdot$&	$\cdot$&	1&	&	$\cdot$&	1&	$\cdot$&	$\cdot$&	$\cdot$&	$\cdot$&	1&	$\cdot$&	$\cdot$&	$\cdot$&	$\cdot$&	1\\
$\cdot$&	$\cdot$&	$\cdot$&	$\cdot$&	1&	$\cdot$&	1&	$\cdot$&	$\cdot$&	$\cdot$&	$\cdot$&	1&	&	$\cdot$&	$\cdot$&	$\cdot$&	1&	$\cdot$&	$\cdot$&	$\cdot$&	1&	$\cdot$&	$\cdot$&	1&	$\cdot$\\
$\cdot$&	$\cdot$&	$\cdot$&	$\cdot$&	$\cdot$&	1&	$\cdot$&	1&	$\cdot$&	$\cdot$&	1&	$\cdot$&	&	$\cdot$&	$\cdot$&	$\cdot$&	$\cdot$&	1&	$\cdot$&	$\cdot$&	$\cdot$&	1&	1&	$\cdot$&	$\cdot$\\
$\cdot$&	$\cdot$&	1&	1&	$\cdot$&	$\cdot$&	1&	$\cdot$&	$\cdot$&	$\cdot$&	$\cdot$&	$\cdot$&	&	$\cdot$&	$\cdot$&	$\cdot$&	1&	$\cdot$&	1&	$\cdot$&	$\cdot$&	$\cdot$&	1&	$\cdot$&	$\cdot$\\
$\cdot$&	$\cdot$&	$\cdot$&	1&	1&	$\cdot$&	$\cdot$&	1&	$\cdot$&	$\cdot$&	$\cdot$&	$\cdot$&	&	$\cdot$&	$\cdot$&	1&	$\cdot$&	1&	$\cdot$&	1&	$\cdot$&	$\cdot$&	$\cdot$&	$\cdot$&	$\cdot$\\
$\cdot$&	$\cdot$&	$\cdot$&	$\cdot$&	1&	1&	$\cdot$&	$\cdot$&	1&	$\cdot$&	$\cdot$&	$\cdot$&	&	$\cdot$&	1&	$\cdot$&	$\cdot$&	$\cdot$&	1&	$\cdot$&	1&	$\cdot$&	$\cdot$&	$\cdot$&	$\cdot$\\
1&	$\cdot$&	$\cdot$&	$\cdot$&	$\cdot$&	1&	$\cdot$&	$\cdot$&	$\cdot$&	1&	$\cdot$&	$\cdot$&	&	1&	$\cdot$&	$\cdot$&	$\cdot$&	1&	$\cdot$&	$\cdot$&	$\cdot$&	$\cdot$&	$\cdot$&	1&	$\cdot$\\
1&	1&	$\cdot$&	$\cdot$&	$\cdot$&	$\cdot$&	$\cdot$&	$\cdot$&	$\cdot$&	$\cdot$&	1&	$\cdot$&	&	$\cdot$&	1&	$\cdot$&	1&	$\cdot$&	$\cdot$&	$\cdot$&	$\cdot$&	1&	$\cdot$&	$\cdot$&	$\cdot$\\
$\cdot$&	1&	1&	$\cdot$&	$\cdot$&	$\cdot$&	$\cdot$&	$\cdot$&	$\cdot$&	$\cdot$&	$\cdot$&	1&	&	$\cdot$&	$\cdot$&	1&	$\cdot$&	$\cdot$&	$\cdot$&	$\cdot$&	$\cdot$&	$\cdot$&	$\cdot$&	$\cdot$&	$\cdot$\\ \cline{1-12} \cline{14-25}
\end{tabular}
\end{center}
\end{scriptsize}
\caption{Principal submatrices $M_S$ of biplanes of order 7.}\label{Fig1}
\end{figure}

Among biplanes of order 9 there is the structure ${\cal B}_{9e}$ with large automorphism group, $|Aut({\cal B}_{9e})|=80640$, admitting both symmetry and trace equal to the number of points. The parameters of arising designs ${\cal D}''(3)$ are 
\begin{eqnarray*}
2-(16,3,\lambda_i), \enspace \lambda_i=0,1,2.
\end{eqnarray*}
The tables on Figure \ref{Fig2} present submatrices $M''_S$ for ${\cal B}_{9e}$. These are incidence matrices of ${\cal D}''(3)$.

\begin{figure}[h!]
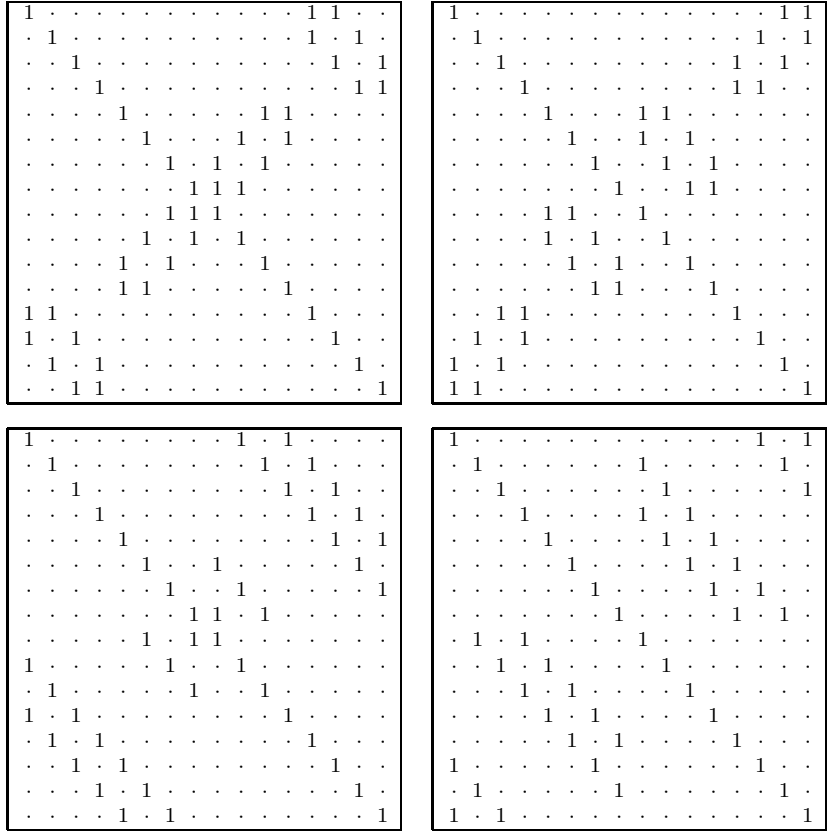

\begin{scriptsize}
\begin{center}

\begin{tabular}{c@{\hspace{0.55em}}c@{\hspace{0.55em}}|c@{\hspace{0.55em}}c@{\hspace{0.55em}}c@{\hspace{0.55em}}c@{\hspace{0.55em}}c@{\hspace{0.55em}}c@{\hspace{0.55em}}c@{\hspace{0.55em}}c@{\hspace{0.55em}}c@{\hspace{0.55em}}c@{\hspace{0.55em}}c@{\hspace{0.55em}}c@{\hspace{0.55em}}c@{\hspace{0.55em}}c@{\hspace{0.55em}}c@{\hspace{0.55em}}c@{\hspace{0.55em}}|c|c@{\hspace{0.55em}}c@{\hspace{0.55em}}c@{\hspace{0.55em}}c@{\hspace{0.55em}}c@{\hspace{0.55em}}c@{\hspace{0.55em}}c@{\hspace{0.55em}}c@{\hspace{0.55em}}c@{\hspace{0.55em}}c@{\hspace{0.55em}}c@{\hspace{0.55em}}c@{\hspace{0.55em}}c@{\hspace{0.55em}}c@{\hspace{0.55em}}c@{\hspace{0.55em}}c@{\hspace{0.55em}}|} \cline{3-18} \cline{20-35}
&&1&	$\cdot$&	$\cdot$&	$\cdot$&	$\cdot$&	$\cdot$&	$\cdot$&	$\cdot$&	$\cdot$&	$\cdot$&	$\cdot$&	$\cdot$&	1&	1&	$\cdot$&	$\cdot$&	&	1&	$\cdot$&	$\cdot$&	$\cdot$&	$\cdot$&	$\cdot$&	$\cdot$&	$\cdot$&	$\cdot$&	$\cdot$&	$\cdot$&	$\cdot$&	$\cdot$&	$\cdot$&	1&	1\\
&&$\cdot$&	1&	$\cdot$&	$\cdot$&	$\cdot$&	$\cdot$&	$\cdot$&	$\cdot$&	$\cdot$&	$\cdot$&	$\cdot$&	$\cdot$&	1&	$\cdot$&	1&	$\cdot$&	&	$\cdot$&	1&	$\cdot$&	$\cdot$&	$\cdot$&	$\cdot$&	$\cdot$&	$\cdot$&	$\cdot$&	$\cdot$&	$\cdot$&	$\cdot$&	$\cdot$&	1&	$\cdot$&	1\\
&&$\cdot$&	$\cdot$&	1&	$\cdot$&	$\cdot$&	$\cdot$&	$\cdot$&	$\cdot$&	$\cdot$&	$\cdot$&	$\cdot$&	$\cdot$&	$\cdot$&	1&	$\cdot$&	1&	&	$\cdot$&	$\cdot$&	1&	$\cdot$&	$\cdot$&	$\cdot$&	$\cdot$&	$\cdot$&	$\cdot$&	$\cdot$&	$\cdot$&	$\cdot$&	1&	$\cdot$&	1&	$\cdot$\\
&&$\cdot$&	$\cdot$&	$\cdot$&	1&	$\cdot$&	$\cdot$&	$\cdot$&	$\cdot$&	$\cdot$&	$\cdot$&	$\cdot$&	$\cdot$&	$\cdot$&	$\cdot$&	1&	1&	&	$\cdot$&	$\cdot$&	$\cdot$&	1&	$\cdot$&	$\cdot$&	$\cdot$&	$\cdot$&	$\cdot$&	$\cdot$&	$\cdot$&	$\cdot$&	1&	1&	$\cdot$&	$\cdot$\\
&&$\cdot$&	$\cdot$&	$\cdot$&	$\cdot$&	1&	$\cdot$&	$\cdot$&	$\cdot$&	$\cdot$&	$\cdot$&	1&	1&	$\cdot$&	$\cdot$&	$\cdot$&	$\cdot$&	&	$\cdot$&	$\cdot$&	$\cdot$&	$\cdot$&	1&	$\cdot$&	$\cdot$&	$\cdot$&	1&	1&	$\cdot$&	$\cdot$&	$\cdot$&	$\cdot$&	$\cdot$&	$\cdot$\\
&&$\cdot$&	$\cdot$&	$\cdot$&	$\cdot$&	$\cdot$&	1&	$\cdot$&	$\cdot$&	$\cdot$&	1&	$\cdot$&	1&	$\cdot$&	$\cdot$&	$\cdot$&	$\cdot$&	&	$\cdot$&	$\cdot$&	$\cdot$&	$\cdot$&	$\cdot$&	1&	$\cdot$&	$\cdot$&	1&	$\cdot$&	1&	$\cdot$&	$\cdot$&	$\cdot$&	$\cdot$&	$\cdot$\\
&&$\cdot$&	$\cdot$&	$\cdot$&	$\cdot$&	$\cdot$&	$\cdot$&	1&	$\cdot$&	1&	$\cdot$&	1&	$\cdot$&	$\cdot$&	$\cdot$&	$\cdot$&	$\cdot$&	&	$\cdot$&	$\cdot$&	$\cdot$&	$\cdot$&	$\cdot$&	$\cdot$&	1&	$\cdot$&	$\cdot$&	1&	$\cdot$&	1&	$\cdot$&	$\cdot$&	$\cdot$&	$\cdot$\\
&&$\cdot$&	$\cdot$&	$\cdot$&	$\cdot$&	$\cdot$&	$\cdot$&	$\cdot$&	1&	1&	1&	$\cdot$&	$\cdot$&	$\cdot$&	$\cdot$&	$\cdot$&	$\cdot$&	&	$\cdot$&	$\cdot$&	$\cdot$&	$\cdot$&	$\cdot$&	$\cdot$&	$\cdot$&	1&	$\cdot$&	$\cdot$&	1&	1&	$\cdot$&	$\cdot$&	$\cdot$&	$\cdot$\\
&&$\cdot$&	$\cdot$&	$\cdot$&	$\cdot$&	$\cdot$&	$\cdot$&	1&	1&	1&	$\cdot$&	$\cdot$&	$\cdot$&	$\cdot$&	$\cdot$&	$\cdot$&	$\cdot$&	&	$\cdot$&	$\cdot$&	$\cdot$&	$\cdot$&	1&	1&	$\cdot$&	$\cdot$&	1&	$\cdot$&	$\cdot$&	$\cdot$&	$\cdot$&	$\cdot$&	$\cdot$&	$\cdot$\\
&&$\cdot$&	$\cdot$&	$\cdot$&	$\cdot$&	$\cdot$&	1&	$\cdot$&	1&	$\cdot$&	1&	$\cdot$&	$\cdot$&	$\cdot$&	$\cdot$&	$\cdot$&	$\cdot$&	&	$\cdot$&	$\cdot$&	$\cdot$&	$\cdot$&	1&	$\cdot$&	1&	$\cdot$&	$\cdot$&	1&	$\cdot$&	$\cdot$&	$\cdot$&	$\cdot$&	$\cdot$&	$\cdot$\\
&&$\cdot$&	$\cdot$&	$\cdot$&	$\cdot$&	1&	$\cdot$&	1&	$\cdot$&	$\cdot$&	$\cdot$&	1&	$\cdot$&	$\cdot$&	$\cdot$&	$\cdot$&	$\cdot$&	&	$\cdot$&	$\cdot$&	$\cdot$&	$\cdot$&	$\cdot$&	1&	$\cdot$&	1&	$\cdot$&	$\cdot$&	1&	$\cdot$&	$\cdot$&	$\cdot$&	$\cdot$&	$\cdot$\\
&&$\cdot$&	$\cdot$&	$\cdot$&	$\cdot$&	1&	1&	$\cdot$&	$\cdot$&	$\cdot$&	$\cdot$&	$\cdot$&	1&	$\cdot$&	$\cdot$&	$\cdot$&	$\cdot$&	&	$\cdot$&	$\cdot$&	$\cdot$&	$\cdot$&	$\cdot$&	$\cdot$&	1&	1&	$\cdot$&	$\cdot$&	$\cdot$&	1&	$\cdot$&	$\cdot$&	$\cdot$&	$\cdot$\\
&&1&	1&	$\cdot$&	$\cdot$&	$\cdot$&	$\cdot$&	$\cdot$&	$\cdot$&	$\cdot$&	$\cdot$&	$\cdot$&	$\cdot$&	1&	$\cdot$&	$\cdot$&	$\cdot$&	&	$\cdot$&	$\cdot$&	1&	1&	$\cdot$&	$\cdot$&	$\cdot$&	$\cdot$&	$\cdot$&	$\cdot$&	$\cdot$&	$\cdot$&	1&	$\cdot$&	$\cdot$&	$\cdot$\\
&&1&	$\cdot$&	1&	$\cdot$&	$\cdot$&	$\cdot$&	$\cdot$&	$\cdot$&	$\cdot$&	$\cdot$&	$\cdot$&	$\cdot$&	$\cdot$&	1&	$\cdot$&	$\cdot$&	&	$\cdot$&	1&	$\cdot$&	1&	$\cdot$&	$\cdot$&	$\cdot$&	$\cdot$&	$\cdot$&	$\cdot$&	$\cdot$&	$\cdot$&	$\cdot$&	1&	$\cdot$&	$\cdot$\\
&&$\cdot$&	1&	$\cdot$&	1&	$\cdot$&	$\cdot$&	$\cdot$&	$\cdot$&	$\cdot$&	$\cdot$&	$\cdot$&	$\cdot$&	$\cdot$&	$\cdot$&	1&	$\cdot$&	&	1&	$\cdot$&	1&	$\cdot$&	$\cdot$&	$\cdot$&	$\cdot$&	$\cdot$&	$\cdot$&	$\cdot$&	$\cdot$&	$\cdot$&	$\cdot$&	$\cdot$&	1&	$\cdot$\\
&&$\cdot$&	$\cdot$&	1&	1&	$\cdot$&	$\cdot$&	$\cdot$&	$\cdot$&	$\cdot$&	$\cdot$&	$\cdot$&	$\cdot$&	$\cdot$&	$\cdot$&	$\cdot$&	1&	&	1&	1&	$\cdot$&	$\cdot$&	$\cdot$&	$\cdot$&	$\cdot$&	$\cdot$&	$\cdot$&	$\cdot$&	$\cdot$&	$\cdot$&	$\cdot$&	$\cdot$&	$\cdot$&	1\\ \cline{3-18} \cline{20-35}
\\ \cline{3-18} \cline{20-35}
&&1&	$\cdot$&	$\cdot$&	$\cdot$&	$\cdot$&	$\cdot$&	$\cdot$&	$\cdot$&	$\cdot$&	1&	$\cdot$&	1&	$\cdot$&	$\cdot$&	$\cdot$&	$\cdot$&	&	1&	$\cdot$&	$\cdot$&	$\cdot$&	$\cdot$&	$\cdot$&	$\cdot$&	$\cdot$&	$\cdot$&	$\cdot$&	$\cdot$&	$\cdot$&	$\cdot$&	1&	$\cdot$&	1\\
&&$\cdot$&	1&	$\cdot$&	$\cdot$&	$\cdot$&	$\cdot$&	$\cdot$&	$\cdot$&	$\cdot$&	$\cdot$&	1&	$\cdot$&	1&	$\cdot$&	$\cdot$&	$\cdot$&	&	$\cdot$&	1&	$\cdot$&	$\cdot$&	$\cdot$&	$\cdot$&	$\cdot$&	$\cdot$&	1&	$\cdot$&	$\cdot$&	$\cdot$&	$\cdot$&	$\cdot$&	1&	$\cdot$\\
&&$\cdot$&	$\cdot$&	1&	$\cdot$&	$\cdot$&	$\cdot$&	$\cdot$&	$\cdot$&	$\cdot$&	$\cdot$&	$\cdot$&	1&	$\cdot$&	1&	$\cdot$&	$\cdot$&	&	$\cdot$&	$\cdot$&	1&	$\cdot$&	$\cdot$&	$\cdot$&	$\cdot$&	$\cdot$&	$\cdot$&	1&	$\cdot$&	$\cdot$&	$\cdot$&	$\cdot$&	$\cdot$&	1\\
&&$\cdot$&	$\cdot$&	$\cdot$&	1&	$\cdot$&	$\cdot$&	$\cdot$&	$\cdot$&	$\cdot$&	$\cdot$&	$\cdot$&	$\cdot$&	1&	$\cdot$&	1&	$\cdot$&	&	$\cdot$&	$\cdot$&	$\cdot$&	1&	$\cdot$&	$\cdot$&	$\cdot$&	$\cdot$&	1&	$\cdot$&	1&	$\cdot$&	$\cdot$&	$\cdot$&	$\cdot$&	$\cdot$\\
&&$\cdot$&	$\cdot$&	$\cdot$&	$\cdot$&	1&	$\cdot$&	$\cdot$&	$\cdot$&	$\cdot$&	$\cdot$&	$\cdot$&	$\cdot$&	$\cdot$&	1&	$\cdot$&	1&	&	$\cdot$&	$\cdot$&	$\cdot$&	$\cdot$&	1&	$\cdot$&	$\cdot$&	$\cdot$&	$\cdot$&	1&	$\cdot$&	1&	$\cdot$&	$\cdot$&	$\cdot$&	$\cdot$\\
&&$\cdot$&	$\cdot$&	$\cdot$&	$\cdot$&	$\cdot$&	1&	$\cdot$&	$\cdot$&	1&	$\cdot$&	$\cdot$&	$\cdot$&	$\cdot$&	$\cdot$&	1&	$\cdot$&	&	$\cdot$&	$\cdot$&	$\cdot$&	$\cdot$&	$\cdot$&	1&	$\cdot$&	$\cdot$&	$\cdot$&	$\cdot$&	1&	$\cdot$&	1&	$\cdot$&	$\cdot$&	$\cdot$\\
&&$\cdot$&	$\cdot$&	$\cdot$&	$\cdot$&	$\cdot$&	$\cdot$&	1&	$\cdot$&	$\cdot$&	1&	$\cdot$&	$\cdot$&	$\cdot$&	$\cdot$&	$\cdot$&	1&	&	$\cdot$&	$\cdot$&	$\cdot$&	$\cdot$&	$\cdot$&	$\cdot$&	1&	$\cdot$&	$\cdot$&	$\cdot$&	$\cdot$&	1&	$\cdot$&	1&	$\cdot$&	$\cdot$\\
&&$\cdot$&	$\cdot$&	$\cdot$&	$\cdot$&	$\cdot$&	$\cdot$&	$\cdot$&	1&	1&	$\cdot$&	1&	$\cdot$&	$\cdot$&	$\cdot$&	$\cdot$&	$\cdot$&	&	$\cdot$&	$\cdot$&	$\cdot$&	$\cdot$&	$\cdot$&	$\cdot$&	$\cdot$&	1&	$\cdot$&	$\cdot$&	$\cdot$&	$\cdot$&	1&	$\cdot$&	1&	$\cdot$\\
&&$\cdot$&	$\cdot$&	$\cdot$&	$\cdot$&	$\cdot$&	1&	$\cdot$&	1&	1&	$\cdot$&	$\cdot$&	$\cdot$&	$\cdot$&	$\cdot$&	$\cdot$&	$\cdot$&	&	$\cdot$&	1&	$\cdot$&	1&	$\cdot$&	$\cdot$&	$\cdot$&	$\cdot$&	1&	$\cdot$&	$\cdot$&	$\cdot$&	$\cdot$&	$\cdot$&	$\cdot$&	$\cdot$\\
&&1&	$\cdot$&	$\cdot$&	$\cdot$&	$\cdot$&	$\cdot$&	1&	$\cdot$&	$\cdot$&	1&	$\cdot$&	$\cdot$&	$\cdot$&	$\cdot$&	$\cdot$&	$\cdot$&	&	$\cdot$&	$\cdot$&	1&	$\cdot$&	1&	$\cdot$&	$\cdot$&	$\cdot$&	$\cdot$&	1&	$\cdot$&	$\cdot$&	$\cdot$&	$\cdot$&	$\cdot$&	$\cdot$\\
&&$\cdot$&	1&	$\cdot$&	$\cdot$&	$\cdot$&	$\cdot$&	$\cdot$&	1&	$\cdot$&	$\cdot$&	1&	$\cdot$&	$\cdot$&	$\cdot$&	$\cdot$&	$\cdot$&	&	$\cdot$&	$\cdot$&	$\cdot$&	1&	$\cdot$&	1&	$\cdot$&	$\cdot$&	$\cdot$&	$\cdot$&	1&	$\cdot$&	$\cdot$&	$\cdot$&	$\cdot$&	$\cdot$\\
&&1&	$\cdot$&	1&	$\cdot$&	$\cdot$&	$\cdot$&	$\cdot$&	$\cdot$&	$\cdot$&	$\cdot$&	$\cdot$&	1&	$\cdot$&	$\cdot$&	$\cdot$&	$\cdot$&	&	$\cdot$&	$\cdot$&	$\cdot$&	$\cdot$&	1&	$\cdot$&	1&	$\cdot$&	$\cdot$&	$\cdot$&	$\cdot$&	1&	$\cdot$&	$\cdot$&	$\cdot$&	$\cdot$\\
&&$\cdot$&	1&	$\cdot$&	1&	$\cdot$&	$\cdot$&	$\cdot$&	$\cdot$&	$\cdot$&	$\cdot$&	$\cdot$&	$\cdot$&	1&	$\cdot$&	$\cdot$&	$\cdot$&	&	$\cdot$&	$\cdot$&	$\cdot$&	$\cdot$&	$\cdot$&	1&	$\cdot$&	1&	$\cdot$&	$\cdot$&	$\cdot$&	$\cdot$&	1&	$\cdot$&	$\cdot$&	$\cdot$\\
&&$\cdot$&	$\cdot$&	1&	$\cdot$&	1&	$\cdot$&	$\cdot$&	$\cdot$&	$\cdot$&	$\cdot$&	$\cdot$&	$\cdot$&	$\cdot$&	1&	$\cdot$&	$\cdot$&	&	1&	$\cdot$&	$\cdot$&	$\cdot$&	$\cdot$&	$\cdot$&	1&	$\cdot$&	$\cdot$&	$\cdot$&	$\cdot$&	$\cdot$&	$\cdot$&	1&	$\cdot$&	$\cdot$\\
&&$\cdot$&	$\cdot$&	$\cdot$&	1&	$\cdot$&	1&	$\cdot$&	$\cdot$&	$\cdot$&	$\cdot$&	$\cdot$&	$\cdot$&	$\cdot$&	$\cdot$&	1&	$\cdot$&	&	$\cdot$&	1&	$\cdot$&	$\cdot$&	$\cdot$&	$\cdot$&	$\cdot$&	1&	$\cdot$&	$\cdot$&	$\cdot$&	$\cdot$&	$\cdot$&	$\cdot$&	1&	$\cdot$\\
&&$\cdot$&	$\cdot$&	$\cdot$&	$\cdot$&	1&	$\cdot$&	1&	$\cdot$&	$\cdot$&	$\cdot$&	$\cdot$&	$\cdot$&	$\cdot$&	$\cdot$&	$\cdot$&	1&	&	1&	$\cdot$&	1&	$\cdot$&	$\cdot$&	$\cdot$&	$\cdot$&	$\cdot$&	$\cdot$&	$\cdot$&	$\cdot$&	$\cdot$&	$\cdot$&	$\cdot$&	$\cdot$&	1\\  \cline{3-18} \cline{20-35}
\end{tabular}

\end{center}
\end{scriptsize}
\caption{Symmetric partially balanced incomplete design with association scheme of three classes arisign from the biplane ${\cal B}_{9e}$, $|Aut({\cal B}_{9e})|=80640$.}\label{Fig2}
\end{figure}
In many cases a matrix $M''_S$ is symmetric in respect to the minor diagonal as well. Althought biplanes of order 7 do not allow symmetry of incidence matrices, equal or analogue patterns in these submatrices can be found within biplanes of both orders, 7 and 9. 

When say a matrix of type $L$, we mean either $L_n$ (that is {\it adjacency matrix} of a {\it path} with starting and ending vertex having {\it loop}) or $L_nC_n^{-}$.  As these tables show, matrices of type $L$ that serves in our proof later we found in the constracted representatives of the family (\ref{NewClass}). It is worth noting that bigger matrices of type $L$ are not admissible with these biplanes. This follows because of the constraints toward canonical part of a biplane's incidence matrix, with similar arguments as we used in Lemma \ref{Lem1}. Due to the same facts, $M_S$ of a biplane of order 11 possibly includes $L_5$ but no bigger $L$-type matrix. 

The next admissible biplane's order 14 (it is the open question if any biplane of this order exists), is the first one where $L$-matrix of maximal size (of order $n-3$) is admissible. More precisely, the matrix 
$$
\begin{bmatrix}
I_{13} & L_{13}C_{13}^{-}
\cr L_{13}C_{13}^{-} & I_{13}
\end{bmatrix}
$$
is admissible as $M_S$ for biplanes of this order. 

Let matrix $T_n$ of order $n$ be defined on the way that 1s are on the position as follows while the other entries are 0,
$$
\qquad
T_{n}=\begin{bmatrix}
 & 1 &... & 1 &
\cr   1&  &  & &1
\cr   \vdots&  &  & & \vdots
\cr   1&  &  & &1
\cr   & 1 & ... & 1&
\end{bmatrix}. 
$$

The first left table of Figure \ref{Fig2} representing ${\cal D}''(3)$ defines the scheme ${\cal A}''$. The next (0,1)-matrices $A_i$, $i \in \{0,1,2,3\}$ are associate matrices of ${\cal A}''$.
%Then for the matrix  ${\cal R}''$ representing the scheme  is presented by the next table:
%$$
%{\cal R}'' =
%\qquad
%\begin{bmatrix}
%T_4+J_4 & J_4 & J_4 & 2L_4C_4^{-}+J4
%\cr J_4 & T_4+J_4 & 2L_4& J_4
%\cr J_4 & 2L_4 & T_4+J4 & J_4
%\cr 2L_4C_4^{-}+J4 & J_4 & J_4 & T_4+J_4
%\end{bmatrix}. 
%$$

\begin{eqnarray*}
A_0&=&I_{16}\\
A_1&=&\begin{bmatrix}
C_4^{-} & J_4 & J_4 & L_4
\cr J_4 & C_4^{-} &  L_4C_4^{-} & J_4
\cr J_4 &  L_4C_4^{-} & C_4^{-}  & J_4
\cr L_4 &  J_4 & J_4  & C_4^{-}
\end{bmatrix}\\
A_2&=&\begin{bmatrix}
T_4 & 0_4 & 0_4 & 0_4
\cr  0_4 & T_4 & 0_4 & 0_4
\cr  0_4 &  0_4 & T_4 & 0_4
\cr  0_4 &  0_4 &  0_4  & T_4  
\end{bmatrix}\\
A_3&=&\begin{bmatrix}
0_4 & 0_4 & 0_4 & L_4C_4^{-}
\cr  0_4 & 0_4 &  L_4 & 0_4
\cr  0_4 &   L_4 & 0_4 & 0_4
\cr   L_4C_4^{-} &  0_4 &  0_4  & 0_4  
\end{bmatrix}
\end{eqnarray*}

\begin{remark} It is not known if there is infinitely or finitely many biplanes. The family of partially balanced incomplete design with association scheme (\ref{NewClass}) is either finite or infinite depending on this answer for biplanes with the two declared properties. 
\end{remark}

\begin{corollary} There is an infinite class of partially balanced incomplete designs with association scheme of three classes, with parameters
\begin{eqnarray} \label{InfClass}
2-(n,3,\lambda_i), \enspace \lambda_i=0,1,2,
\end{eqnarray}
for $n \ge 6$.
\end{corollary}

\begin{proof}
For every $n\ge 6$ there is at least one representative of this class, the matrix $D_n$ being its incidence matrix. 
\end{proof}

Note that he family \ref{NewClass} is a subclass of class \ref{InfClass}.

It is worth to remind that biplanes are structures similar to finite projective planes. While within projective planes any two lines intersect in one point, with biplanes two lines intersect in two points. Having in mind that certain discrete structures arise form projective planes, no wonder that the same is case for biplanes. Results presented here additinally confirm that biplanes as very regular structures. We showed that biplanes with the two described properties include partially balanced designs as well as association schemes, that further form vector spaces and algebras.

\section*{Acknowledgement}

Author is thankful to Professor Dragutin Svrtan form the University of Zagreb for advices on this subject, especially related to Lemma \ref{Lem1}.

\end{document}